\documentclass[11pt,a4paper]{amsart}

\usepackage{graphicx,amssymb}
\usepackage[top=40mm,bottom=35mm,left=30mm,right=30mm]{geometry}
\usepackage[frame]{xy}
\xyoption{arc}

\def\baselinestretch{1.2}

\theoremstyle{plan}
\newtheorem{theorem}{Theorem}[section]
\newtheorem{lemma}[theorem]{Lemma}
\newtheorem{corollary}[theorem]{Corollary}
\newtheorem{proposition}[theorem]{Proposition}

\theoremstyle{definition}
\newtheorem{definition}[theorem]{Definition}
\newtheorem{example}[theorem]{Example}
\newtheorem{remark}[theorem]{Remark}

\begin{document}

\title[Path lifting properties and embedding between RAAGs]
{Path lifting properties\\ and embedding between RAAGs}

\author{Eon-Kyung Lee and Sang-Jin Lee}

\address{Department of Mathematics, Sejong University, Seoul, 143-747, Korea}

\email{eonkyung@sejong.ac.kr}

\address{Department of Mathematics, Konkuk University, Seoul, 143-701, Korea}

\email{sangjin@konkuk.ac.kr}

\date{\today}

\begin{abstract}
For a finite simplicial graph $\Gamma$, let $G(\Gamma)$ denote
the right-angled Artin group on the complement graph of $\Gamma$.
In this article, we introduce the notions of
``induced path lifting property'' and ``semi-induced path lifting property''
for immersions between graphs,
and obtain graph theoretical criteria for the embedability
between right-angled Artin groups.
We recover the result of S.-h.{} Kim and T.{} Koberda
that an arbitrary  $G(\Gamma)$ admits a quasi-isometric
group embedding into $G(T)$ for some finite tree $T$.
The upper bound on the number of vertices of $T$ is improved from
$2^{2^{(m-1)^2}}$ to $m2^{m-1}$,
where $m$ is the number of vertices of $\Gamma$.
We also show that the upper bound on the number of vertices of $T$
is at least $2^{m/4}$.
Lastly, we show that $G(C_m)$ embeds in $G(P_n)$ for $n\geqslant 2m-2$,
where $C_m$ and $P_n$ denote the cycle and path graphs
on $m$ and $n$ vertices, respectively.

\end{abstract}

\maketitle

\section{Introduction}\label{sec:intro}

For a finite simplicial graph $\Gamma$, let $G(\Gamma)$ denote
the right-angled Artin group on the complement graph of $\Gamma$.

This paper is motivated by the recent work~\cite{KK13b} of
S.-h.{} Kim and T.{} Koberda that
if $\Gamma$ is a graph with $m$ vertices
then $G(\Gamma)$ is embeddable into $G(T)$ for some
tree $T$ with at most $2^{2^{(m-1)^2}}$ vertices,
hence $G(\Gamma)$ is embeddable into the $n$-strand braid group $B_n$
with $n \leqslant 2^{2^{m^2}}$.
This result has several important corollaries.

At the first glance of their paper, we thought that the double
exponential upper bound for the braid index is far from being sharp
(one may expect a polynomial upper bound)
and that their construction is very interesting and instructive,
but it is not simple enough for practical uses.

We have tried to make a new construction which gives
a polynomial upper bound on the number of vertices of the tree $T$.
Though we were short of this goal,
we succeeded in improving their results as follows.
First, the double exponential upper bound $2^{2^{(m-1)^2}}$
is improved to an exponential upper bound $m2^m$.
Second, we show that an exponential upper bound
is unavoidable as long as the embedding $G(\Gamma)\to G(T)$
is induced by an immersion $T\to\Gamma$
as in the construction of Kim and Koberda.

These come from graph theoretical criteria on immersions between graphs
for the embeddability between associated right-angled Artin groups,
using \emph{path lifting properties} of immersions.
We think that these properties are simple enough for practical uses.
In particular, we show that
$G(C_n)$ is embeddable into $G(P_{2n-2})$,
where $C_m$ and $P_m$ denote the cycle and path graphs
on $m$ vertices, respectively.
This generalizes the result of
M.{} Casals-Ruiz, A.{} Duncan and I.{} Kazachkov~\cite{CDK13}
that $G(C_5)$ is embeddable in $G(P_8)$.

\subsection{Right-angled Artin groups}

Throughout the paper, all graphs are assumed to be undirected and
simplicial (that is, without loops or multiple edges).
Let $\Gamma$ and $\Gamma_1$ be graphs.
We denote by $V(\Gamma)$ and $E(\Gamma)$
the vertex and edge sets of $\Gamma$, respectively.
The notation ``$\Gamma_1 \leqslant \Gamma$'' means that
$\Gamma_1$ is an induced subgraph of $\Gamma$,
that is, $V(\Gamma_1) \subseteq V(\Gamma)$ and
$E(\Gamma_1)=\{\, \{v_1,v_2\}\in E(\Gamma)\mid v_1,v_2\in V(\Gamma_1) \,\}$.
For a vertex $v\in V(\Gamma)$, the \emph{link} of $v$ in $\Gamma$ is the set
${\operatorname{Lk}}_\Gamma(v)=\{\, u\in V(\Gamma)\mid \{v,u\}\in E(\Gamma) \,\}$.
For $A \subseteq V(\Gamma)$, we denote by $\Gamma{\backslash} A$
the induced subgraph of $\Gamma$ on $V(\Gamma){\backslash} A$.

For a finite graph $\Gamma$,
the \emph{right-angled Artin group} (RAAG) $A(\Gamma)$ on $\Gamma$ is defined by the presentation
$$ A(\Gamma)=\langle\, v\in V(\Gamma)\mid [v_i,v_j]=1\
\mbox{if $\{v_i,v_j\}\in E(\Gamma)$}\,\rangle.$$
In the present paper, we use the opposite convention
$$ G(\Gamma)=\langle\, v\in V(\Gamma)\mid [v_i,v_j]=1\
\mbox{if $\{v_i,v_j\}\not\in E(\Gamma)$}\,\rangle. $$
In other words, $G(\Gamma)=A(\Gamma^c)$, where $\Gamma^c$ denotes the complement graph of $\Gamma$.

For metric spaces $(X,d_X)$ and $(Y,d_Y)$, a map $f:X\to Y$ is a \emph{quasi-isometric embedding}
if there is a constant $C\geqslant 1$ such that for any $x_1,x_2\in X$
$$
d_X(x_1,x_2)/C -C \leqslant d_Y(f(x_1),f(x_2)) \leqslant  C d_X(x_1,x_2)+C.
$$
For finitely generated groups $G$ and $H$, a \emph{quasi-isometric group embedding from $G$ to $H$}
is an injective homomorphism $f:G\to H$ such that it is a quasi-isometric embedding when $G$
and $H$ are endowed with word-metrics.

\subsection{Embeddability of RAAGs into braid groups}

An interesting question in the theory of RAAGs is,  given a graph $\Gamma$,
for which surface $S$ the group $G(\Gamma)$ admits an embedding into the mapping class group ${\operatorname{Mod}}(S)$. (See~\cite{CW07, CLM12,KK13b,KK14} for instance.)
It is well-known that if there is an embedding of the graph $\Gamma$ into $S_{g,p}$
(the orientable surface with genus $g$ and $p$ punctures),
then $G(\Gamma)$ admits a quasi-isometric embedding into ${\operatorname{Mod}}(S_{g,q})$
for some $q\geqslant p$.
For example, if $\Gamma$ is a planar graph, then $G(\Gamma)$ is
embeddable into the mapping class group of a punctured sphere.
Notice that, in this construction, the genus of the surface is at least
the graph-genus of $\Gamma$ (i.e.\ the smallest genus of a surface $S$
which admits an embedding $\Gamma\hookrightarrow S$).
Therefore one may ask the following question.

\medskip

\noindent\textbf{Question.}\ \ 
For a finite (non-planar) graph $\Gamma$, does $G(\Gamma)$ admit
a quasi-isometric embedding into the mapping class group of a punctured sphere?

\medskip

It was answered affirmatively by S.-h.{} Kim and T.{} Koberda~\cite{KK13b}.

\begin{theorem}[\cite{KK13b}]\label{thm:KK14}
For each finite graph $\Gamma$, there exists a finite tree $T$ such that
$G(\Gamma)$ admits a quasi-isometric group embedding into $G(T)$.
\end{theorem}

This theorem has several interesting corollaries:
any RAAG admits quasi-isomorphic group embeddings into a pure braid group
and into the area-preserving diffeomorphism groups of the 2-disk and the 2-sphere;
every finite-volume hyperbolic 3-manifold group is virtually a
quasi-isometrically embedded subgroup of a pure braid group.

\subsection{The Kim-Koberda construction}
Let us briefly review the construction of Kim and Koberda.

A \emph{map of graphs} $\phi:\Lambda\to\Gamma$ consists of a pair of functions,
vertices to vertices and edges to edges, preserving the structure,
i.e.\ $\phi$ sends adjacent vertices to adjacent vertices.

Let $\phi:\Lambda\to\Gamma$ be a map of graphs between finite graphs.
Then $\phi$ induces a group homomorphism $\phi^*:G(\Gamma)\to G(\Lambda)$
defined by
$$\phi^*(v)=\prod\limits_{v'\in \phi^{-1}(v)} v'$$
for $v\in V(\Gamma)$, where the product is defined to be the identity
if $\phi^{-1}(v)$ is the empty set.
Since $\phi$ is a map of graphs and since $\Gamma$ has no loops,
no two vertices of $\phi^{-1}(v)$ are adjacent, hence the product is well-defined.

We say that $\phi:\Lambda\to\Gamma$ is \emph{$F$-surviving} for $F \subseteq V(\Lambda)$
if for any $v'\in F$ and for any reduced word $w$ in $G(\Gamma)$,
the word $\phi^*(w)$ has no cancellation of $v'$. (See \S2 for details.)
It is not difficult to observe that
if $\phi$ is $F$-surviving for some $F \subseteq V(\Lambda)$ with $\phi(F)=V(\Gamma)$
then $\phi^*:G(\Gamma)\to G(\Lambda)$ is a quasi-isometric group embedding.

In~\cite{KK13b}, Kim and Koberda considered
the universal cover $p:\tilde \Gamma\to\Gamma$,
a finite induced subgraph $T \leqslant \tilde \Gamma$,
and the restriction $\phi=p|_T:T\to \Gamma$.
Here, we may assume that $\Gamma$ and $T$ are connected, hence $T$ is a tree.
They showed that
\emph{if $T$ is sufficiently large,
then $\phi$ is $F$-surviving for some $F \subseteq V(T)$ with $\phi(F)=V(\Gamma)$,
hence $\phi^*:G(\Gamma)\to G(T)$ is a quasi-isometric group embedding.}
In order to achieve $\phi$ being $F$-surviving, they repeatedly enlarged
the tree $T$ by taking union with its images
under a certain collection of ``deck transformations''.
They showed that one can take $T$ such that
$|V(T)| \leqslant 2^{2^{(m-1)^2}}$,
where $m=|V(\Gamma)|$.

\subsection{Path lifting properties of immersions and our results}

A map of graphs $\phi:\Lambda\to\Gamma$
is called an \emph{immersion} (or a \emph{locally injective map})
if the restriction $\phi|_{{\operatorname{Lk}}_\Lambda(v')}$
is injective for each $v'\in V(\Lambda)$.

Notice that a map of graphs $\phi:\Lambda\to\Gamma$ is a \emph{covering}
if $\phi$ is surjective and
$\phi|_{{\operatorname{Lk}}_\Lambda(v')}:{\operatorname{Lk}}_\Lambda(v')\to {\operatorname{Lk}}_\Gamma(\phi(v'))$ is bijective
for each $v'\in V(\Lambda)$.
It is easy to see that,
for finite simplicial graphs $\Lambda$ and $\Gamma$,
a map of graphs $\phi:\Lambda\to\Gamma$ is an immersion
if and only if there is a covering $p:\tilde\Gamma\to\Gamma$
that extends $\phi$, i.e.
$\Lambda$ is an induced subgraph of $\tilde\Gamma$
such that $\phi=p|_\Lambda$.
(For example, J.~Stallings~\cite[Theorem 6.1]{Sta83} proved this
for the case when $\Gamma$ is a bouquet.
Our case can be proved similarly.)

As observed by J.~Stallings~\cite{Sta83},
immersions $\phi:\Lambda\to\Gamma$ have some of the properties
of coverings such as the unique path lifting property
and $\pi_1$-injectivity.
However, a path in $\Gamma$ may not be lifted
to a path in $\Lambda$.

In \S\ref{sec:PL}, we introduce the notions of
semi-induced paths and induced paths in $\Gamma$.
We say that a map of graphs $\phi:\Lambda\to\Gamma$ has
the \emph{semi-induced path lifting property} (SIPL) for
$F \subseteq V(\Lambda)$
if any semi-induced path in $\Gamma$ starting from $\phi(v')$
for some $v'\in F$
is lifted to a path in $\Lambda$ starting from $v'$.
The \emph{induced path lifting property} (IPL)
is defined similarly.
(See \S\ref{sec:PL} for details.)

\smallskip

\noindent\textbf{Theorem~\ref{thm:sipl2surv}.}\ \emph{Let $\phi:\Lambda\to\Gamma$ 
be an immersion between finite graphs,
and let $F$ be a finite nonempty subset of $V(\Lambda)$.
If $\phi$ has SIPL for $F$, then $\phi$ is $F$-surviving.}

\smallskip

The property SIPL plays the role of deck transformation in the Kim-Koberda construction.
Using this, we prove the following.

\smallskip

\noindent\textbf{Theorem~\ref{main:upperbd}.}\ \emph{For each graph $\Gamma$ with $m$ vertices,
there exists a tree $T$ with $|V(T)| \leqslant m2^{m-1}$ such that
$G(\Gamma)$ admits a quasi-isometric group embedding into $G(T)$.}

\smallskip

Therefore the upper bound on $|V(T)|$ is improved from a double exponential
function in $m=|V(\Gamma)|$ in~\cite{KK13b} to an exponential function.

The following theorem shows that a partial converse of
Theorem~\ref{thm:sipl2surv} holds.

\smallskip

\noindent\textbf{Theorem~\ref{thm:ipl}.}\ \emph{Let $\phi:\Lambda\to \Gamma$ be an immersion between finite graphs,
and let $F$ be a finite nonempty subset of $V(\Lambda)$.
If $\phi$ is $F$-surviving, then $\phi$ has IPL for $F$.}

\smallskip

Using this, we show that the upper bound on $|V(T)|$ in Theorem~\ref{main:upperbd}
must be at least an exponential function in $m$.

\smallskip

\noindent\textbf{Theorem~\ref{main:lowerbd}.}\ \emph{
For each integer $m\geqslant 2$, there exists a graph  $\Gamma_m$ with $m$ vertices
such that if $\phi:T\to\Gamma_m$ is an immersion of a finite tree $T$ into $\Gamma_m$
and $\phi^*:G(\Gamma_m)\to G(T)$ is injective,
then $|V(T)| \geqslant 2^{m/4}$.}

\smallskip

The graphs $\Gamma_m$ in the above theorem is
thanks to Young Soo Kwon, Sang-il Oum and Paul Seymour.
The lower bound from our original example was $|V(T)| \geqslant  2^{\sqrt m}$.

\smallskip
Let $C_m$ and $P_n$ denote the cycle and path graphs
on $m$ and $n$ vertices, respectively.
Regard $P_n$ as a subgraph of the universal cover $\tilde C_m$ of $C_m$
which is the bi-infinite path graph,
and let $\phi_{n,m}:P_n\to C_m$ be the restriction of the covering map
$\tilde C_m\to C_m$.
In \cite{CDK13}, M.{} Casals-Ruiz, A.{} Duncan and I.{} Kazachkov
showed that $\phi_{8,5}^*: G(C_5)\to G(P_8)$ is injective,
which gives a counterexample to the Weakly Chordal Conjecture in~\cite{KK13a}.
(In~\cite{CDK13}, $P_n$ denotes the path graph of length $n$,
hence it is $P_{n+1}$ in our notation.)
Using the induced path lifting property, we establish the following.

\smallskip

\noindent\textbf{Theorem~\ref{main:cdk}.}\ 
\emph{For each $m\geqslant 3$, $\phi_{n,m}^*:G(C_m)\to G(P_n)$ is injective
if and only if $n\geqslant 2m-2$.}

\smallskip

We close this section with a couple of remarks.

\begin{remark}
Let $\phi:\Lambda\to\Gamma$ be an immersion between finite graphs and $F \subseteq V(\Lambda)$.
By Theorems~\ref{thm:sipl2surv} and~\ref{thm:ipl}, we know that
\begin{quote}
$\phi$ has SIPL for $F$
$\Rightarrow$ $\phi$ is $F$-surviving
$\Rightarrow$ $\phi$ has IPL for $F$.
\end{quote}
It would be interesting to know whether the converses hold.
\end{remark}

\begin{remark}
Suppose that we are given a quasi-isometric group embedding $\phi^*:G(\Gamma)\to G(T)$ for a tree $T$.
Then, as observed in~\cite{KK13b}, $G(T)$ admits a quasi-isometric group
embedding into the pure braid group $P_n$
with $n \leqslant 4|V(T)|+2$, hence so does the original group $G(\Gamma)$.
For the embeddability into pure braid groups,
it suffices to require $T$ to be a planar graph, not necessarily a tree.
Hence we can use a planar cover $\tilde \Gamma$ rather than a universal cover,
and it would give a smaller upper bound on $|V(T)|$.
For example, if $\Gamma$ is embeddable into a M\"obius band (equivalently, in a real projective plane), then
its double cover $\tilde\Gamma$ is a planar graph.
It would be interesting to see whether the upper bound obtained by using planar cover
is substantially smaller than the one given in this paper.
\end{remark}

\section{Preliminaries}\label{sec:prelim}

For a map of graphs $\phi:\Lambda\to\Gamma$
and induced subgraphs $\Lambda_1 \leqslant \Lambda$ and $\Gamma_1 \leqslant \Gamma$
with $\phi(\Lambda_1) \subseteq \Gamma_1$,
we denote by $\phi(\Lambda_1,\Gamma_1)$
the restriction $\phi|_{\Lambda_1}:\Lambda_1\to\Gamma_1$.

As noted in \S\ref{sec:intro}, a map of graphs  $\phi:\Lambda\to\Gamma$
induces a well-defined group homomorphism
$\phi^*:G(\Gamma)\to G(\Lambda)$ defined by
$\phi^*(v)=\prod_{v'\in \phi^{-1}(v)} v'$ for $v\in V(\Gamma)$.
Abusing notation, for a word $w$ in $G(\Gamma)$, $\phi^*(w)$ denotes
the word defined by the product.
For this, we may fix a total order on $V(\Lambda)$
and write each product $\prod_{v'\in \phi^{-1}(v)} v'$
in the increasing order.

The map of graphs $\phi$ considered in~\cite{KK13b}
is a restriction of the universal cover $p:\tilde \Gamma\to\Gamma$
to a subtree $T\leqslant \tilde \Gamma$.
However, the group homomorphism $\phi^*$ can be defined
for any map of graphs $\phi:\Lambda\to\Gamma$,
and most of their arguments work for immersions or regular covers.
Hence we describe their construction
in a little more general setting.

Lemmas~\ref{lem:deck} and~\ref{lem:link-s} are the key lemmas in this section.
Corollary~\ref{cor:link-surv} is a simple version of Lemma~\ref{lem:link-s}.
The tree $T$ in~\cite{KK13b}
is constructed by repeatedly applying Lemma~\ref{lem:deck} and Corollary~\ref{cor:link-surv}.
So they are somehow implicit in the proof of Lemma 11 in~\cite{KK13b}.
Lemma~\ref{lem:link-s} is an improved version of Corollary~\ref{cor:link-surv},
which was inspired by examples in~\cite{CDK13}.

\smallskip

Let $w$ be a word in $G(\Gamma)$ representing $g \in G(\Gamma)$.
The word $w$ is \emph{reduced} if $w$ is a ``shortest'' word among all words representing $g$.
In this case, the length of $w$ is the \emph{word length} of $g$, denoted by $\Vert g\Vert$.
The \emph{support} of $w$, denoted by ${\operatorname{supp}}(w)$, is the set of all vertices $v\in V(\Gamma)$ such that
$v$ or $v^{-1}$ appears in $w$.
It is well-known that if $w'$ is another reduced word representing $g$, then
${\operatorname{supp}}(w)={\operatorname{supp}}(w')$.

Let $w$ be a (non-reduced) word in  $G(\Gamma)$.
A subword $v^{\pm1}w_1v^{\mp1}$ of $w$ is a \emph{cancellation} of $v$ in $w$
if ${\operatorname{supp}}(w_1)\cap {\operatorname{Lk}}_\Gamma(v)=\emptyset$.
If, furthermore, no letter in $w_1$ is equal to $v$ or $v^{-1}$,
it is an \emph{innermost cancellation} of $v$ in $w$.
It is known that $w$ is reduced if and only if $w$ has no (innermost) cancellation.

\begin{definition}
We say that $\phi:\Lambda\to\Gamma$ is \emph{$v'$-surviving} for $v'\in V(\Lambda)$
if for any reduced word $w$ in $G(\Gamma)$,
the word $\phi^*(w)$ has no innermost cancellation of $v'$.
We say that $\phi$ is \emph{$F$-surviving} for $F \subseteq V(\Lambda)$
if it is $v'$-surviving for each $v'\in F$.
\end{definition}

For example, let $\phi:\Lambda\to \Gamma$ be as in Figure~\ref{fig:square}.
Then $\phi^*(v_1v_4v_1^{-1})=v_1'v_1''v_4'v_1''^{-1}v_1'^{-1}$
has an innermost cancellation of $v_1'$,
hence $\phi$ is not $v_1'$-surviving.
In fact, $\phi^*:G(\Gamma)\to G(\Lambda)$ is not injective:
$w=v_1v_2v_1^{-1}v_4v_1v_2^{-1}v_1^{-1}v_4^{-1}$ is a reduced word in $G(\Gamma)$
but $\phi^*(w)$ is the identity in $G(\Lambda)$.

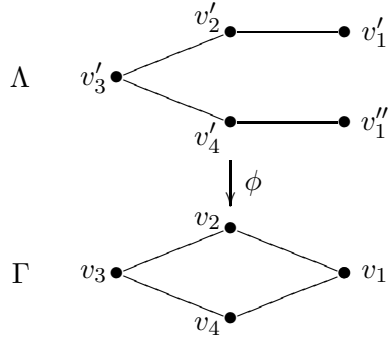
\begin{figure}
$$\begin{xy}
(30, 0) *{\bullet} ;
(15, 6) *{\bullet}  **@{-} ;
(0,  0) *{\bullet}  **@{-} ;
(15,-6) *{\bullet}  **@{-} ;
(30, 0) *{\bullet}  **@{-} ;
(34, 0) *{v_1};
(12,  7) *{v_2};
(-3, 0) *{v_3};
(12, -7) *{v_4};
(-10,0) *+!R{\Gamma};
(30,32) *{\bullet};
(15,32) *{\bullet}  **@{-};
(0, 26) *{\bullet}  **@{-};
(15,20) *{\bullet}  **@{-};
(30,20) *{\bullet}  **@{-};
(34,32) *{v_1'};
(34,20) *{v_1''};
(12,34) *{v_2'};
(12,18) *{v_4'};
(-3,26) *{v_3'};
(-10,26) *+!R{\Lambda};
(15,9); (15,15) **@{-}    ?<*@{<}   ?(0.5) *!/^3mm/{\phi};
\end{xy}
$$
\caption{$\phi$ maps $v_i'$ to $v_i$ for all $i$ and $v_1''$ to $v_1$. }
\label{fig:square}
\end{figure}

\begin{lemma}[{\cite[Lemma 9]{KK13b}}]\label{lem:KK9}
Let $\phi_1:\Lambda_1\to\Gamma$ be a map of graphs.
If $\Lambda\leqslant\Lambda_1$ and $\phi=\phi_1|_{\Lambda}$,
then ${\operatorname{supp}}(\phi^*(w)) \subseteq \operatorname{supp}(\phi_1^*(w))$ for any $w\in G(\Gamma)$.
In particular, $\ker\phi_1^* \subseteq \ker\phi^*$.
\end{lemma}

\begin{lemma}[{\cite[Lemma 10]{KK13b}}]\label{lem:KK10}
If a map of graphs $\phi:\Lambda\to\Gamma$ is $F$-surviving for some $F \subseteq  V(\Lambda)$ with $\phi(F)=V(\Gamma)$,
then $\phi^*:G(\Gamma)\to G(\Lambda)$ is a quasi-isometric group embedding.
\end{lemma}

The idea of proof of Lemma~\ref{lem:KK9} is as follows:
the inclusion map $\iota:\Lambda\to \Lambda_1$ induces
a homomorphism $\iota^*:G(\Lambda_1)\to G(\Lambda)$ such that
$\phi^*=\iota^*\circ\phi_1^*$ and $\iota^*$
sends the vertices in $\Lambda_1{\backslash}\Lambda$ to the identity.

The idea of proof of Lemma~\ref{lem:KK10} is as follows:
if $\phi:\Lambda\to\Gamma$ is $F$-surviving for some $F \subseteq V(\Lambda)$ with $\phi(F)=V(\Gamma)$,
then $\Vert\phi^*(w)\Vert \geqslant \Vert w\Vert$ for any reduced word $w$ in $G(\Gamma)$.

The following lemma shows that, under a certain condition,
if $\phi:\Lambda\to\Gamma$ is an immersion such that $\phi^*$ is injective,
then $\phi$ can be extended to $\phi_1:\Lambda_1\to\Gamma$
such that $\phi_1$ is $F$-surviving.

\begin{lemma}\label{lem:deck}
Let $\phi:\Lambda\to\Gamma$ be an immersion
that is a restriction of a regular cover
$p:\tilde\Gamma\to\Gamma$
to a finite induced subgraph $\Lambda$ of $\tilde\Gamma$.
Suppose that $\phi$ is surjective on the sets of vertices.
Let $F$ be a finite nonempty subset of $V(\tilde\Gamma)$,
and let $\Sigma$ be the set of all deck transformations
$\sigma:\tilde\Gamma\to\tilde\Gamma$ such that
$\sigma(\Lambda)\cap F\ne\emptyset$.
Let $\Lambda_1$ be the induced subgraph of
$\tilde\Gamma$ on
$$\bigcup_{\sigma\in\Sigma} \sigma(V(\Lambda)).$$
If $\phi^*:G(\Gamma)\to G(\Lambda)$ is injective,
then $\phi_1=p|_{\Lambda_1}:\Lambda_1\to\Gamma$ is $F$-surviving.
\end{lemma}

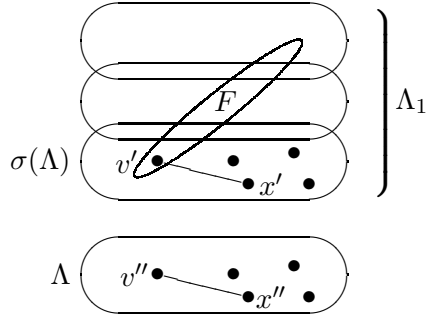
\begin{figure}
$$
\begin{xy}
( 0, 0) *{\bullet}; (12,-3) *{\bullet}  **@{-} ;
(10, 0) *{\bullet};
(18, 1) *{\bullet};
(20,-3) *{\bullet};
(-3, 0) *{v''};
(15,-3) *{x''};
(-10, -5); (25,5) **\frm<20pt>{-};
(-10,0) *+!R{\Lambda};
( 0, 15) *{\bullet}; (12, 12) *{\bullet} **@{-};
(10, 15) *{\bullet};
(18, 16) *{\bullet};
(20, 12) *{\bullet};
(-4, 15) *{v'};
(15, 12) *{x'};
(-10,15) *+!R{\sigma(\Lambda)};
(-3, 13); (8,22), {\ellipse<,2mm>{-}};
( 9, 23) *+{F};
(-10, 10); (25,20) **\frm<20pt>{-};
(-10, 18); (25,28) **\frm<20pt>{-};
(-10, 26); (25,36) *{}, **\frm<20pt>{-};
(30,11);(30,35) **\frm{)};
(30,23) *+!L{\txt{$\Lambda_1$}};
\end{xy}
$$
\caption{The deck transformation $\sigma$ sends
$v''$ and $x''$ to $v'$ and $x'$, respectively.}
\label{fig:deck}
\end{figure}

\begin{proof}
Notice that $F \subseteq  V(\Lambda_1)$
because $\phi$ is surjective and $p$ is a regular cover.
Assume that $\phi_1$ is not $v'$-surviving for some $v'\in F$.
Let $v=\phi_1(v')$.
See Figure~\ref{fig:deck}.
Let $w$ be a nontrivial reduced word in $G(\Gamma)$
such that $\phi_1^*(w)$ has an innermost cancellation of $v'$.
Then $w$ has a subword of the form
$$v^{\pm 1}w_1v^{\mp1}$$
such that $w_1$ is a word in $G(\Gamma{\backslash} v)$
with ${\operatorname{supp}}(\phi_1^*(w_1))\cap{\operatorname{Lk}}_{\Lambda_1}(v')=\emptyset$.

Since $\phi^*$ is injective, $\phi^*(v^{\pm 1}w_1v^{\mp1})\ne\phi^*(w_1)$.
Hence there exists $v''\in \phi^{-1}(v) \subseteq V(\Lambda)$ such that
$v''\in{\operatorname{supp}}(\phi^*(v^{\pm 1}w_1v^{\mp1}))$,
thus there exists
$$x''\in{\operatorname{supp}}(\phi^*(w_1))\cap{\operatorname{Lk}}_\Lambda(v'').$$
Because $p:\tilde\Gamma\to\Gamma$ is a regular cover, there is a
deck transformation $\sigma$ such that $\sigma(v'')=v'$.
(Here, $\sigma\in\Sigma$ because $v'=\sigma(v'')\in \sigma(\Lambda)\cap F$.)
Let $x'=\sigma(x'')$.
By Lemma~\ref{lem:KK9},
\begin{align*}
x'=\sigma(x'')
&\in \sigma({\operatorname{supp}}(\phi^*(w_1)))\cap {\operatorname{Lk}}_{\sigma(\Lambda)}(\sigma(v''))\\
&= {\operatorname{supp}}((\phi\circ\sigma^{-1})^*(w_1))\cap {\operatorname{Lk}}_{\sigma(\Lambda)}(v')\\
& \subseteq  \operatorname{supp}(\phi_1^*(w_1))\cap {\operatorname{Lk}}_{\Lambda_1}(v'),
\end{align*}
which contradicts ${\operatorname{supp}}(\phi_1^*(w_1))\cap{\operatorname{Lk}}_{\Lambda_1}(v')=\emptyset$.
\end{proof}

\begin{remark}\label{rmk:deck-size}
In Lemma~\ref{lem:deck}, if $\Gamma$ is connected, then
$$|\Sigma| \leqslant |V(\Lambda)|\cdot
\max \left\{\, |p^{-1}(v)\cap F|:v\in V(\Gamma) \, \right\}
$$
because a deck transformation $\sigma\in\Sigma$ is uniquely determined by a vertex
$v'\in V(\Lambda)$ and its image $\sigma(v')\in p^{-1}(v)\cap F$, where $v=p(v')$.
In particular, if $|p^{-1}(v)\cap F| \leqslant 1$ for all $v\in V(\Gamma)$,
then $|\Sigma| \leqslant |V(\Lambda)|$, hence
$$|V(\Lambda_1)| \leqslant |\Sigma|\cdot|V(\Lambda)| \leqslant |V(\Lambda)|^2.$$
\end{remark}

\begin{lemma}\label{lem:link-s}
Let $\phi:\Lambda\to\Gamma$ be an immersion.
Let $v'$ be a vertex of $\Lambda$ such that $\phi({\operatorname{Lk}}_\Lambda(v'))={\operatorname{Lk}}_\Gamma(v)$,
where $v=\phi(v')$.
Let ${\operatorname{Lk}}_\Gamma(v)=\{x_1,\ldots,x_l\}$ and ${\operatorname{Lk}}_\Lambda(v')=\{x_1',\ldots,x_l'\}$
such that $\phi(x_i')=x_i$ for  $1 \leqslant i \leqslant l$.
Let
\begin{align*}
\Gamma_i  &=\Gamma{\backslash} \{v,x_1,\ldots,x_{i-1}\},\\
\Lambda_i &=\Lambda{\backslash} \phi^{-1}(\{v,x_1,\ldots,x_{i-1}\}),\\
\phi_i    &=\phi(\Lambda_i,\Gamma_i):\Lambda_i\to\Gamma_i
\end{align*}
for $1 \leqslant i \leqslant l$.
If each $\phi_i$ is $x_i'$-surviving, then $\phi$ is $v'$-surviving.
Furthermore, if $\phi_1^*:G(\Gamma_1)\to G(\Lambda_1)$ is injective,
then so is $\phi^*:G(\Gamma)\to G(\Lambda)$.
\end{lemma}

\begin{figure}
$$\begin{xy}
( 0, 0) *{\bullet}; (10,-3) *{\bullet} **@{-} ;
( 0, 0) *{\bullet}; (20,-1) *{\bullet} **@{-} ;
( 0, 0) *{\bullet}; (25, 1) *{\bullet} **@{-} ;
( 0, 0) *{\bullet}; (35, 4) *{\bullet} **@{-} ;
(-3, 0) *{v};
(11,-5) *{x_1};
(14,-2.6) *{\cdot};(15,-2.3) *{\cdot};(16,-2) *{\cdot};
(29, 2) *{\cdot};(30,2.3) *{\cdot};(31, 2.6) *{\cdot};
(20,-3) *{x_{i-1}};
(26,-1) *{x_i};
(36, 2) *{x_l};
(40,-2) *{\bullet};
(45,0)  *{\bullet};
(23.5, -7); (50,7) **\frm<8pt>{-};
(50, 0) *+!L{\txt{$\Gamma_i$}};
( 0, 25) *{\bullet}; (10, 20) *{\bullet} **@{-} ;
( 0, 25) *{\bullet}; (20, 23) *{\bullet} **@{-} ;
( 0, 25) *{\bullet}; (25, 25) *{\bullet} **@{-} ;
( 0, 25) *{\bullet}; (35, 30) *{\bullet} **@{-} ;
(-3, 25) *{v'};
(11, 18) *{x_1'};
(20, 21) *{x_{i-1}'};
(26, 23) *{x_i'};
(36, 28) *{x_l'};
( 0, 30) *{\bullet}; ( 0, 33) *{\bullet};
(10, 31) *{\bullet}; (10, 34) *{\bullet};
(20, 32) *{\bullet}; (20, 35) *{\bullet};
(25, 33) *{\bullet}; (25, 36) *{\bullet};
(35, 18) *{\bullet}; (35, 21) *{\bullet};
(40, 19) *{\bullet}; (40, 22) *{\bullet};
(45, 21) *{\bullet}; (45, 24) *{\bullet};
(23.5, 16); (50,38) **\frm<8pt>{-};
(50, 25) *+!L{\txt{$\Lambda_i$}};
(14, 21) *{\cdot};(15,21.3) *{\cdot};(16,21.6) *{\cdot};
(29, 27) *{\cdot};(30,27.3) *{\cdot};(31, 27.6) *{\cdot};
(30,8); (30,15) **@{-}    ?<*@{<}   ?(0.5) *!/^3mm/{\phi_i};
\end{xy}
$$
\caption{The map of graphs $\phi_i$}
\label{fig:link}
\end{figure}
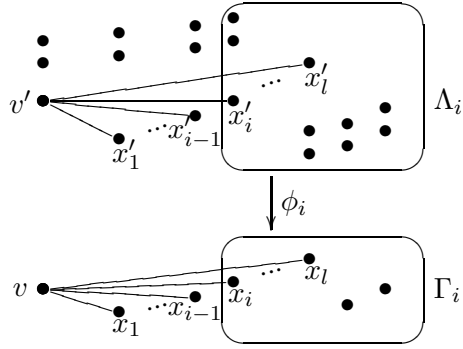

\begin{proof}
See Figure~\ref{fig:link}.
Assume that $\phi$ is not $v'$-surviving.
Then there exists a nontrivial reduced word $w$ in $G(\Gamma)$
with an innermost cancellation of $v'$ in $\phi^*(w)$,
hence there is a subword of $w$ of the form
$$
v^{\pm 1}w_1v^{\mp1}
$$
such that $w_1$ is a word in $G(\Gamma{\backslash} v)$ with
${\operatorname{supp}}(\phi^*(w_1))\cap{\operatorname{Lk}}_\Lambda(v')=\emptyset$.

Since $w$ is reduced, ${\operatorname{supp}}(w_1)\cap{\operatorname{Lk}}_\Gamma(v)$ is nonempty.
Let
$$i=\min\{k: x_k\in {\operatorname{supp}}(w_1)\cap{\operatorname{Lk}}_\Gamma(v)\}.$$
Then $w_1$ is a reduced word in $G(\Gamma_i)$,
hence $\phi^*(w_1)=\phi_i^*(w_1)$ and
${\operatorname{supp}}(\phi^*(w_1))={\operatorname{supp}}(\phi_i^*(w_1))$.

Since $x_i\in {\operatorname{supp}}(w_1)$ and $\phi_i$ is $x_i'$-surviving, $x_i'\in{\operatorname{supp}}(\phi_i^*(w_1))={\operatorname{supp}}(\phi^*(w_1))$.
Therefore $x_i'\in{\operatorname{supp}}(\phi^*(w_1))\cap {\operatorname{Lk}}_\Lambda(v')$,
which contradicts ${\operatorname{supp}}(\phi^*(w_1))\cap{\operatorname{Lk}}_\Lambda(v')=\emptyset$.
Therefore $\phi$ is $v'$-surviving.

Now, suppose that $\phi_1^*$ is injective.
Let $w$ be a nontrivial reduced word in $G(\Gamma)$.
If $v\not\in{\operatorname{supp}}(w)$, then $w$ belongs to $G(\Gamma_1)=G(\Gamma{\backslash} v)$,
hence $\phi^*(w)=\phi_1^*(w)$ is nontrivial in $G(\Lambda_1)$.
Since $\Lambda_1$ is an induced subgraph of $\Lambda$, $G(\Lambda_1)$ embeds in $G(\Lambda)$,
hence $\phi^*(w)$ is nontrivial in $G(\Lambda)$.
If $v\in{\operatorname{supp}}(w)$, then $\phi^*(w)$ is nontrivial because
$\phi$ is $v'$-surviving.
Therefore $\phi^*$ is injective.
\end{proof}

In the above lemma, observe that, for each $1 \leqslant i \leqslant l$ and a reduced word
$w$ in $G(\Gamma_i)$, if $\phi_i^*(w)$ has a cancellation of $x_i'$ then
so does $\phi_1^*(w)$.
Thus if $\phi_1$ is $x_i'$-surviving, then so is $\phi_i$.
Hence we have the following simple version.

\begin{corollary}\label{cor:link-surv}
Let $\phi:\Lambda\to\Gamma$ be an immersion.
Let $v'$ be a vertex of $\Lambda$ such that $\phi({\operatorname{Lk}}_\Lambda(v'))={\operatorname{Lk}}_\Gamma(v)$,
where $v=\phi(v')$. Let
$$
\Gamma_1  =\Gamma{\backslash} v,\quad
\Lambda_1 =\Lambda{\backslash} \phi^{-1}(v),\quad
\phi_1    =\phi(\Lambda_1,\Gamma_1):\Lambda_1\to\Gamma_1.
$$
If $\phi_1$ is ${\operatorname{Lk}}_\Lambda(v')$-surviving, then $\phi$ is $v'$-surviving.
Furthermore, if  $\phi_1^*:G(\Gamma_1)\to G(\Lambda_1)$
is injective, then so is $\phi^*:G(\Gamma)\to G(\Lambda)$.
\end{corollary}

\section{Path lifting properties and embedding between RAAGs}
\label{sec:PL}

In this section we introduce the notions of
``induced path lifting property'' and ``semi-induced path lifting property''
for immersions, and apply them to embedability between RAAGs.
This will improve some results of Kim and Koberda in~\cite{KK13b}
and Casals-Ruiz, Duncan and Kazachkov in~\cite{CDK13}.

Throughout this section, $\phi:\Lambda\to\Gamma$
is assumed to be an immersion between finite graphs.

\subsection{Path lifting properties}

\begin{definition}
A \emph{path} in a graph $\Gamma$ is a tuple $\alpha=(v_0,v_1,\ldots,v_k)$
of vertices of $\Gamma$ such that $\{v_i,v_{i+1}\}\in E(\Gamma)$
for $0 \leqslant i \leqslant k-1$ and $v_i\ne v_j$ if $i\ne j$.
(In particular, $\alpha$ is not a loop because $v_0\ne v_k$.)
\end{definition}

\begin{definition}
A path $\alpha=(v_0,v_1,\ldots,v_k)$ in $\Gamma$ is an \emph{induced path}
if $\{v_i,v_j\}\not\in E(\Gamma)$ whenever $j \geqslant  i+2$.
(In other words, $\alpha$ is an induced path
if and only if the induced subgraph of $\Gamma$
on the vertices of $\alpha$ is a path graph.)
\end{definition}

\begin{definition}
Suppose that a total order $\prec$ is given on $V(\Gamma)$.
A path $\alpha=(v_0,v_1,\ldots,v_k)$ in $\Gamma$ is
a \emph{semi-induced path} with respect to $\prec$ if $\{v_i,v_j\}\not\in E(\Gamma)$
whenever $j \geqslant  i+2$ and $v_j\prec v_{i+1}$.
\end{definition}

If there is no confusion, we will not refer to the total order on $V(\Gamma)$,
assuming that some order is given.
In particular, we omit the term ``with respect to $\prec$'' for simplicity.

\medskip

A path with one or two vertices is both induced and semi-induced.
Notice that a path $\alpha=(v_0,v_1,\ldots,v_k)$ is an induced path
if, for each $j\geqslant 2$,
$$v_j\not\in\bigcup_{i=0}^{j-2}{\operatorname{Lk}}_\Gamma(v_i)$$
and it is a semi-induced path
if, for each $j\geqslant 2$,
$$v_j\not\in\bigcup_{i=0}^{j-2}
\{v\in{\operatorname{Lk}}_\Gamma(v_i):v\prec v_{i+1}\}.
$$

Let $\phi:\Lambda\to\Gamma$ be an immersion.
A path $\tilde\alpha=(v_0',v_1',\ldots,v_k')$ in $\Lambda$
is called a \emph{lift} of the path
$\alpha=(v_0,v_1,\ldots,v_k)$ in $\Gamma$
if $\phi(v_i')=v_i$ for $0 \leqslant i \leqslant k$.
Notice that the lift $\tilde\alpha$, if exists,
is uniquely determined
by $\alpha$ and $v_0'$
due to the unique path lifting property
for immersions between graphs~\cite{Sta83}.

\begin{definition}
Let $\phi:\Lambda\to\Gamma$ be an immersion,
and let $F \subseteq  V(\Lambda)$.
We say that $\phi$ has
the \emph{induced path lifting property (IPL) for $F$}
(resp.\ \emph{semi-induced path lifting property (SIPL) for $F$})
if for any $v'\in F$ and for any induced path
(resp.\ semi-induced path) $\alpha$ in $\Gamma$
starting from $\phi(v')$,
there is a lift of $\alpha$ starting from $v'$.
\end{definition}

It is obvious from the definitions that if $\phi$ has SIPL for $F$,
then $\phi$ has IPL for $F$.

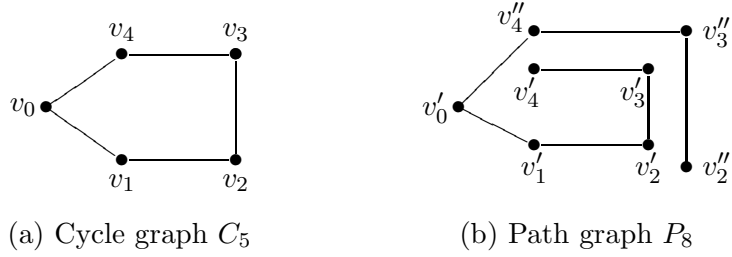
\begin{figure}
\begin{tabular}{*7c}
$\begin{xy}
( 0, 0) *{\bullet};
( 0, 0); (10,-7) *{\bullet} **@{-}; (25,-7) *{\bullet} **@{-}; (25, 7) *{\bullet} **@{-};
( 0, 0); (10, 7) *{\bullet} **@{-}; (25, 7) *{\bullet} **@{-};
(-3, 0) *{v_0};
(10,-10) *{v_1}; (25,-10) *{v_2};
(10, 10) *{v_4}; (25, 10) *{v_3};
\end{xy}$
&\qquad\qquad\qquad&
$\begin{xy}
( 0, 0) *{\bullet};
( 0, 0); (10,-5) *{\bullet} **@{-}; (25,-5) *{\bullet} **@{-};
(25, 5) *{\bullet} **@{-}; (10, 5) *{\bullet} **@{-};
( 0, 0); (10, 10) *{\bullet} **@{-}; (30, 10) *{\bullet} **@{-}; (30, -8) *{\bullet} **@{-};
(-3, 0) *{v_0'};
(10,-8) *{v_1'}; (25,-8) *{v_2'}; (9, 2) *{v_4'}; (23, 2) *{v_3'};
(7,12) *{v_4''}; (34, 10) *{v_3''}; (34, -8) *{v_2''};
\end{xy}$\\[3em]
(a) Cycle graph $C_5$ &&
(b) Path graph $P_8$
\end{tabular}
\caption{The map $\phi:P_8\to C_5$ sends $v_i'$ and $v_i''$
to $v_i$ $(0 \leqslant i \leqslant 4)$.}
\label{fig:C5}
\end{figure}

\begin{example}\label{eg:cdk}
Let $C_5$ be the cycle graph on five vertices as in Figure~\ref{fig:C5}(a).
Then $\alpha_1=(v_0,v_1,v_2,v_3,v_4)$ and $\alpha_2=(v_0,v_4,v_3,v_2,v_1)$
are maximal paths starting from $v_0$,
but they are not induced paths because $\{v_0,v_4\}$ and $\{v_0,v_1\}$ are edges of $C_5$.
The paths $(v_0,v_1,v_2,v_3)$ and $(v_0,v_4,v_3,v_2)$ are
the only maximal induced paths starting from $v_0$.

Fix a total order on $V(C_5)$,
say, $v_0\prec v_1\prec v_2\prec v_3\prec v_4$.
Then $\alpha_2=(v_0,v_4,v_3,v_2,v_1)$ is not semi-induced because
$v_1\in \{v\in {\operatorname{Lk}}_{C_5}(v_0): v\prec v_4\}=\{v_1\}$.
The paths $\alpha_1=(v_0,v_1,v_2,v_3,v_4)$ and $\alpha_3=(v_0,v_4,v_3,v_2)$
are the only maximal semi-induced paths starting from $v_0$.
Similarly, $\alpha_4=(v_1,v_0,v_4,v_3,v_2)$ and $\alpha_5=(v_1,v_2,v_3,v_4)$
are the only maximal semi-induced paths starting from $v_1$.

Let $P_8$ be the path graph on eight vertices as in Figure~\ref{fig:C5}(b),
and let $\phi:P_8\to C_5$ map $v_i'$ and $v_i''$ to $v_i$ for $0 \leqslant i \leqslant 4$.
Then $\phi$ is an immersion.

The maximal induced paths in $C_5$ starting from $v_0$ are
$(v_0,v_1,v_2,v_3)$ and $(v_0,v_4,v_3,v_2)$,
and they can be lifted to paths in $P_8$ starting from $v_0'$.
Hence $\phi$ has IPL for $v_0'$.
In the same way, $\phi$ has IPL and SIPL for $\{v_0',v_1'\}$.
In fact, the map $\phi:P_8\to C_5$ is an example given in~\cite{CDK13} such that
$\phi^*:G(C_5)\to G(P_8)$ is injective.
Theorem~\ref{main:cdk} gives a necessary and sufficient condition on
$(m,n)$ for $\phi:P_n\to C_m$, defined as above, to induce
an injective group homomorphism.
\end{example}

\begin{figure}
\begin{tabular}{*4c}
$\begin{xy}
( 0, 0) *{\bullet};
( 0, 0); (10,-8) *{\bullet} **@{-}; (20,0) *{\bullet} **@{-};
( 0, 0); (10, 0) *{\bullet} **@{-}; (20,0) *{\bullet} **@{-};
( 0, 0); (10, 8) *{\bullet} **@{-}; (20,0) *{\bullet} **@{-};
(-3, 0) *{v_0};
(10,-11) *{v_1};
(10,-3) *{v_2};
(10, 5) *{v_3};
(23, 0) *{v_4};
\end{xy}$
&\qquad&
$\begin{xy}
( 0, 0) *{\bullet};
( 0, 0); (10,-8) *{\bullet} **@{-}; (20,-8) *{\bullet} **@{-}; (30,-10)  *{\bullet} **@{-};
   (20,-8); (30,-6)  *{\bullet} **@{-};
( 0, 0); (10, 0) *{\bullet} **@{-}; (20, 0) *{\bullet} **@{-}; (30,-2)  *{\bullet} **@{-};
   (20, 0); (30, 2)  *{\bullet} **@{-};
( 0, 0); (10, 8) *{\bullet} **@{-}; (20, 8) *{\bullet} **@{-}; (30, 6)  *{\bullet} **@{-};
   (20, 8); (30,10)  *{\bullet} **@{-};
(-3, 0) *{v_0};
(10,-11)*{v_1};
(10,-3) *{v_2};
(10, 5) *{v_3};
(21, 5) *{v_4};
(21,-3) *{v_4};
(21,-11)*{v_4};
(33,10) *{v_2};
(33,6) *{v_1};
(33,2) *{v_3};
(33,-2) *{v_1};
(33,-6) *{v_3};
(33,-10) *{v_2};
\end{xy}$\\[3em]
(a) Graph $\Gamma=K_{2,3}$ &&
(b) Maximal paths from $v_0$\\[2em]
$\begin{xy}
( 0, 0) *{\bullet};
( 0, 0); (10,-8) *{\bullet} **@{-}; (20,-8) *{\bullet} **@{-};
( 0, 0); (10, 0) *{\bullet} **@{-}; (20, 0) *{\bullet} **@{-};
( 0, 0); (10, 8) *{\bullet} **@{-}; (20, 8) *{\bullet} **@{-};
(-3, 0) *{v_0};
(10,-11)*{v_1};
(10,-3) *{v_2};
(10, 5) *{v_3};
(21, 5) *{v_4};
(21,-3) *{v_4};
(21,-11)*{v_4};
\end{xy}$&&
$\begin{xy}
( 0, 0) *{\bullet};
( 0, 0); (10,-8) *{\bullet} **@{-}; (20,-8) *{\bullet} **@{-}; (30,-10)  *{\bullet} **@{-};
   (20,-8); (30,-6)  *{\bullet} **@{-};
( 0, 0); (10, 0) *{\bullet} **@{-}; (20, 0) *{\bullet} **@{-}; (30,0)  *{\bullet} **@{-};
( 0, 0); (10, 8) *{\bullet} **@{-}; (20, 8) *{\bullet} **@{-};
(-3, 0) *{v_0};
(10,-11)*{v_1};
(10,-3) *{v_2};
(10, 5) *{v_3};
(21, 5) *{v_4};
(21,-3) *{v_4};
(21,-11)*{v_4};
(33,0) *{v_3};
(33,-6) *{v_3};
(33,-10) *{v_2};
\end{xy}$
\\[3em]
(c) Maximal induced paths from $v_0$ &&
(d) Maximal semi-induced paths from $v_0$
\end{tabular}
\caption{Paths in $\Gamma=K_{2,3}$ starting from $v_0$. For semi-induced paths, we use the
order on $V(\Gamma)$ given by $v_0\prec v_1\prec v_2\prec v_3$.}
\label{fig:path-K23}
\end{figure}
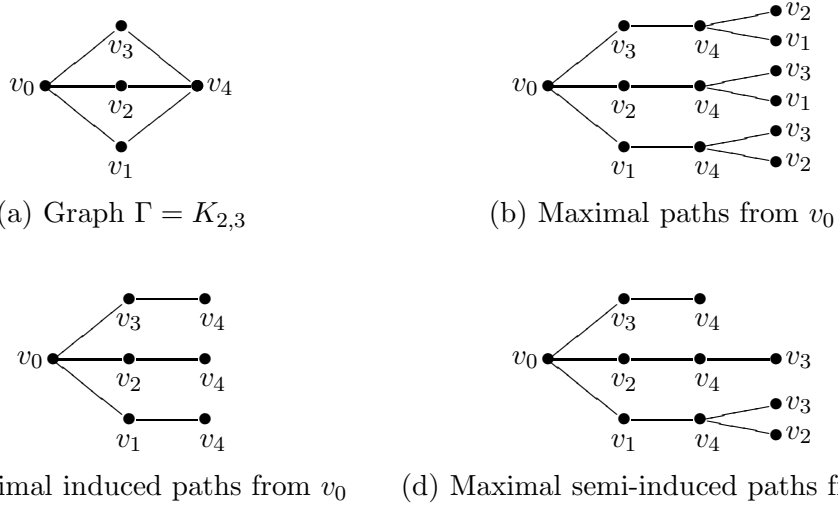

\begin{example}
Let $\Gamma$ be the complete bipartite graph $K_{2,3}$.
See Figure~\ref{fig:path-K23}, which shows
maximal paths, semi-induced paths and induced paths  starting from $v_0$ in $\Gamma$.
\end{example}

\begin{example}
If $\Gamma$ is the complete graph $K_n$ on $n$ vertices endowed with the total order
$v_0\prec v_1\prec\cdots\prec v_{n-1}$,
then $\alpha=(v_{i_0},v_{i_1},\ldots,v_{i_k})$ is a maximal semi-induced path starting from $v_0$
if and only if $0=i_0<i_1<i_2<\cdots<i_{k-1}<i_k=n-1$,
hence there are $2^{n-2}$ maximal semi-induced
paths starting from $v_0$.
\end{example}

\begin{lemma}\label{lem:link-sipl}
Let $\phi:\Lambda\to\Gamma$ be an immersion,
and let $v'$ be a vertex of $\Lambda$ such that
$\phi({\operatorname{Lk}}_\Lambda(v'))={\operatorname{Lk}}_\Gamma(v)$, where $v=\phi(v')$.
Let ${\operatorname{Lk}}_\Gamma(v)=\{x_1,\ldots,x_l\}$
and ${\operatorname{Lk}}_\Lambda(v')=\{x_1',\ldots,x_l'\}$ with $\phi(x_i')=x_i$ for  $1 \leqslant i \leqslant l$.
Suppose that a total order is given on $V(\Gamma)$
such that $x_1\prec x_2\prec\cdots\prec x_l$.
For $1 \leqslant i \leqslant l$, let
\begin{align*}
\Gamma_i  &=\Gamma{\backslash} \{v,x_1,\ldots,x_{i-1}\},\\
\Lambda_i &=\Lambda{\backslash} \phi^{-1}(\{v,x_1,\ldots,x_{i-1}\}),\\
\phi_i    &=\phi(\Lambda_i,\Gamma_i):\Lambda_i\to\Gamma_i.
\end{align*}
Suppose that each $V(\Gamma_i)$ inherits the total order
from $V(\Gamma)$.
Then $\phi:\Lambda\to\Gamma$ has SIPL for $v'$
if and only if
each $\phi_i$ has SIPL for $x_i'$.
\end{lemma}

\begin{proof}
Assume that $\phi$ has SIPL for $v'$. Fix $i\in\{1,\ldots,l\}$.
Let $\alpha_1=(x_i,v_1,\ldots,v_k)$ be a semi-induced path
in $\Gamma_i$ starting from $x_i$.
Then $\alpha=(v,x_i,v_1,\ldots,v_k)$ is a path in $\Gamma$
since $x_i\in{\operatorname{Lk}}_\Gamma(v)$.

Since $\alpha_1$ is a path in $\Gamma_i$, we have
$v_j\not\in\{x_1,\ldots,x_{i-1}\}
=\{u\in {\operatorname{Lk}}_\Gamma(v): u\prec x_i\}$
for $1 \leqslant j \leqslant k$.
Notice that $\alpha_1$ is semi-induced also in $\Gamma$
because $\Gamma_i$ is an induced subgraph of $\Gamma$.
Hence, for $2 \leqslant j \leqslant k$, we have
$v_j\not\in\{u\in {\operatorname{Lk}}_\Gamma(x_i): u\prec v_1\}$ and
$v_j\not\in\{u\in {\operatorname{Lk}}_\Gamma(v_p): u\prec v_{p+1}\}$
for $1 \leqslant p \leqslant j-2$.
These imply that $\alpha$ is a semi-induced path in $\Gamma$.

Since $\phi$ has SIPL for $v'$, there is a lift
$\tilde\alpha=(v',x_i'',v_1',\ldots,v_k')$ starting from $v'$.
Because $(v',x_i')$ is a lift of the path $(v,x_i)$,
the unique path lifting property implies that $x_i''=x_i'$.
Then the restriction $(x_i',v_1',\ldots,v_k')$ of
$\tilde\alpha$ is a lift of $\alpha_1$ to $\Lambda_i$
starting from $x_i'$.
Therefore $\phi_i$ has SIPL for $x_i'$.

Conversely, assume that each $\phi_i$ has SIPL for $x_i'$.
Let $\alpha$ be a semi-induced path in $\Gamma$
starting from $v$, hence $\alpha$ is of the form
$\alpha=(v,x_i,v_1,\ldots,v_k)$ for some $1 \leqslant i \leqslant l$
and $v_1,\ldots,v_k\in V(\Gamma)$.

Let $\alpha_1=(x_i,v_1,\ldots,v_k)$.
Since $\alpha$ is semi-induced in $\Gamma$, we have
$v_j\not\in \{u\in {\operatorname{Lk}}_\Gamma(v): u\prec x_i\}=\{x_1,\ldots,x_{i-1}\}$
for $1 \leqslant j \leqslant k$.
Hence $\alpha_1$ is a path in $\Gamma_i$.
Since $\alpha_1$ is semi-induced in $\Gamma$,
it is semi-induced also in $\Gamma_i$
because $V(\Gamma_i)$ inherits
the total order from $V(\Gamma)$.

Since $\phi_i$ has SIPL for $x_i'$,
there is a lift $\tilde\alpha_1=(x_i',v_1',\ldots,v_k')$
of $\alpha_1$ to $\Lambda_i$ starting from $x_i'$.
Then $\tilde\alpha=(v',x_i',v_1',\ldots,v_k')$
is a lift of $\alpha$ to $\Lambda$ starting from $v'$.
Therefore $\phi$ has SIPL for $v'$.
\end{proof}

\subsection{Semi-induced path lifting property and embedding between RAAGs}

Compare the similarity between Lemma~\ref{lem:link-sipl} and Lemma~\ref{lem:link-s}.
From this, we have the following theorem almost for free.

\begin{theorem}\label{thm:sipl2surv}
Let $\phi:\Lambda\to\Gamma$ be an immersion between finite graphs,
and let $F$ be a finite nonempty subset of $V(\Lambda)$.
If $\phi$ has SIPL for $F$, then $\phi$ is $F$-surviving.
\end{theorem}

\begin{proof}
Choose any $v'\in F$ and let $v=p(v')$.
It suffices to show that $\phi$ is $v'$-surviving.
Notice that $\phi({\operatorname{Lk}}_\Lambda(v'))={\operatorname{Lk}}_\Gamma(v)$
because if $x\in{\operatorname{Lk}}_\Gamma(v)$ then $(v,x)$ is a semi-induced path in $\Gamma$
hence there is a unique lift $(v',x')$ of $(v,x)$ to $\Lambda$.

Let ${\operatorname{Lk}}_\Gamma(v)=\{x_1,\ldots,x_l\}$
and ${\operatorname{Lk}}_\Lambda(v')=\{x_1',\ldots,x_l'\}$ with $\phi(x_i')=x_i$ for  $1 \leqslant i \leqslant l$.
Rearranging $x_i$'s if necessary, we may assume that the total order on $V(\Gamma)$
is such that $x_1\prec x_2\prec\cdots\prec x_l$.
For $1 \leqslant i \leqslant l$, let
\begin{align*}
\Gamma_i  &=\Gamma{\backslash} \{v,x_1,\ldots,x_{i-1}\},\\
\Lambda_i &=\Lambda{\backslash} \phi^{-1}(\{v,x_1,\ldots,x_{i-1}\}),\\
\phi_i    &=\phi(\Lambda_i,\Gamma_i):\Lambda_i\to\Gamma_i.
\end{align*}

By Lemma~\ref{lem:link-sipl}, each $\phi_i$ has SIPL for $x_i'$.
Using induction on the number of vertices of $\Gamma$, we may assume that
each $\phi_i$ is $x_i'$-surviving.
(If $|V(\Gamma)|=1$, then $\phi$ is obviously $v'$-surviving.)
Then $\phi$ is $v'$-surviving by Lemma~\ref{lem:link-s}.
\end{proof}

\begin{proposition}\label{thm:sipl-emb}
Let $\Gamma$ be a graph with $m$ vertices,
$p:\tilde \Gamma\to\Gamma$ a covering,
and $F$ a finite nonempty subset of $V(\tilde\Gamma)$.
Then there exists an induced subgraph $\Lambda$ of $\tilde\Gamma$
with $|V(\Lambda)| \leqslant |F|\cdot 2^{m-1}$
such that $\phi=p|_\Lambda:\Lambda\to\Gamma$ has SIPL for $F$.
Furthermore, if the induced subgraph of $\tilde \Gamma$ on $F$ is connected,
then $\Lambda$ can be chosen to be connected.
\end{proposition}

\begin{proof}
For each $v'\in V(\tilde\Gamma)$, let $\Lambda_{v'}$ be the union of the lifts
of all maximal semi-induced paths in $\Gamma$ starting from  $p(v')$
to $\tilde \Gamma$ starting from $v'$.
Then $p|_{\Lambda_{v'}}:\Lambda_{v'}\to\Gamma$ has SIPL for $v'$.

\medskip

\noindent\emph{Claim.}\ \emph{$|V(\Lambda_{v'})| \leqslant 2^{m-1}$.}

\medskip

\begin{proof}[Proof of Claim]
We use induction on $m=|V(\Gamma)|$.
If $m=1$, then both $\Gamma$ and $\Lambda_{v'}$ are
the graph with one vertex and no edge,
hence the claim holds.

Let ${\operatorname{Lk}}_{\tilde \Gamma}(v')=\{x_1',\ldots,x_l'\}$
and ${\operatorname{Lk}}_{\Gamma}(v)=\{x_1,\ldots,x_l\}$,
where $v=p(v')$ and $x_i=p(x_i')$ for $1 \leqslant i \leqslant l$.
Rearranging $x_i$'s if necessary,
we may assume that the total order on $V(\Gamma)$ is such that
$x_1\prec x_2\prec\cdots\prec x_l$.
For $1 \leqslant i \leqslant l$, let
\begin{align*}
\Gamma_i  &=\Gamma{\backslash} \{v,x_1,\ldots,x_{i-1}\},
\end{align*}
and let $\Lambda_i$ be the union of the lifts
of all maximal semi-induced paths in $\Gamma_i$ starting from  $x_i$
to $\tilde \Gamma$ starting from $x_i'$.
Then $\Lambda_i$ is defined in the same way as $\Lambda_{v'}$,
where $\Gamma_i$ and $x_i'$ play the roles of
$\Gamma$ and $v'$, respectively.
By induction hypothesis,
$|V(\Lambda_i)| \leqslant 2^{m-i-1}$
because $|V(\Gamma_i)|=|V(\Gamma)|-i=m-i$.
Notice that $l \leqslant m-1$ and
$V(\Lambda_{v'})=\{v'\}
\cup\left(\cup_{i=1}^l V(\Lambda_i)\right)$.
Therefore
$$
|V(\Lambda_{v'})|
\leqslant 1+\sum_{i=1}^l |V(\Lambda_i)|
\leqslant 1+\sum_{i=1}^{m-1} 2^{m-i-1}
= 1+\sum_{k=0}^{m-2} 2^k
=2^{m-1}.
$$\vskip -\baselinestretch\baselineskip
\end{proof}

Let $\Lambda$ be the induced subgraph of $\tilde\Gamma$
on the vertices of $\cup_{v'\in F}\Lambda_{v'}$,
and let $\phi=p|_\Lambda:\Lambda\to\Gamma$.
By the construction, $\phi$ has SIPL for $F$ and
$|V(\Lambda)|
\leqslant \sum_{v'\in F} |V(\Lambda_{v'})|
\leqslant |F|\cdot 2^{m-1}$.
It is easy to see that
if the induced subgraph of $\tilde\Gamma$ on $F$ is connected,
then $\Lambda$ is connected.
\end{proof}

In the above theorem, if $p:\tilde\Gamma\to\Gamma$ is the universal cover
and $\Gamma$ is connected, we can take $F$ as the vertex set of
a lift of a maximal tree in $\Gamma$.
Then $|F|=m$ and $\Lambda$ is a tree with
$|V(\Lambda)| \leqslant |F|\cdot 2^{m-1} =m2^{m-1}$.
Combining this observation with Lemma~\ref{lem:KK10}
and Theorem~\ref{thm:sipl2surv}, we have the following.

\begin{corollary}\label{cor:sipl-tree}
If $\Gamma$ is a connected graph with $m$ vertices
and $p:\tilde\Gamma\to\Gamma$ is the universal cover,
then there exists an induced subtree $T$ of $\tilde \Gamma$ with $|V(T)| \leqslant m2^{m-1}$
such that $\phi^*:G(\Gamma)\to G(T)$ is a quasi-isometric group embedding,
where $\phi=p|_T$.
\end{corollary}

Theorem~\ref{main:upperbd} is immediate from the above corollary.

\begin{theorem}\label{main:upperbd}
For each graph $\Gamma$ with $m$ vertices,
there exists a tree $T$ with $|V(T)| \leqslant m2^{m-1}$ such that
$G(\Gamma)$ admits a quasi-isometric group embedding into $G(T)$.
\end{theorem}

\begin{proof}
We follow the argument of Kim and Koberda in~\cite{KK13b}.

If $\Gamma$ is connected, then it is Corollary~\ref{cor:sipl-tree}.

Suppose that $\Gamma$ is disconnected, hence
$\Gamma=\Gamma_1\coprod \Gamma_2$
with $|V(\Gamma_i)|=m_i \geqslant 1$ for $i=1,2$.
Using induction on $|V(\Gamma)|$, we may assume that for each $i=1,2$,
there exist a tree $T_i$ with $|V(T_i)| \leqslant m_i2^{m_i-1}$
and a quasi-isometric group embedding of $G(\Gamma_i)$
into $G(T_i)$.
Let $T$ be the tree obtained by joining a vertex in $T_1$
and another vertex in $T_2$ by a length 2 path.
Since $m_1,m_2 \geqslant 1$,
\begin{align*}
|V(T)|&=|V(T_1)|+|V(T_2)|+1 \leqslant m_12^{m_1-1}+m_22^{m_2-1}+1\\
& \leqslant (m_1+m_2)2^{m_1+m_2-1}=m2^{m-1}.
\end{align*}
Since $G(\Gamma)= G(\Gamma_1)\times G(\Gamma_2)$
and there is a natural quasi-isometric embedding of
$G(T_1)\times G(T_2)$ into $G(T)$,
there is a quasi-isometric group embedding of $G(\Gamma)$ into $G(T)$.
\end{proof}

\begin{remark}
Let us say that an immersion $\phi:\Lambda\to\Gamma$ has
the \emph{path lifting property} (PL) for $F \subseteq V(\Lambda)$
if, for any $v'\in F$ and for any path $\alpha$ in $\Gamma$
starting from $v=\phi(v')$,
there is a lift of $\alpha$ to $\Lambda$ starting from $v'$.

If $\phi$ has PL for $F$, then $\phi$ has SIPL for $F$, by definitions,
regardless of the choice of a particular total order on $V(\Gamma)$,
hence $\phi$ is $F$-surviving by Theorem~\ref{thm:sipl2surv}.

The notion of PL is simpler than SIPL, and PL is enough for some cases.
For example, we can obtain Corollary~\ref{cor:sipl-tree} with
$|V(T)| \leqslant m\cdot m!\,$ if SIPL in Proposition~\ref{thm:sipl-emb} is
replaced with PL; the upper bound on $|V(T)|$ increases
from $m\cdot 2^{m-1}$ to $m\cdot m!$.

However, the property SIPL is indeed necessary for some cases.
For example, the map $\phi:P_8\to C_5$ in Example~\ref{eg:cdk} has SIPL for $v_0'$
but it does not have PL for $v_0'$.
\end{remark}

\subsection{Induced path lifting property and embedding between RAAGs}\label{ssec:ipl}
Recall from Theorem~\ref{main:upperbd} that
for each graph $\Gamma$ with $m$ vertices,
there exists a tree $T$ with
$$|V(T)| \leqslant m2^{m-1}$$
such that
$G(\Gamma)$ admits a quasi-isometric group embedding into $G(T)$.
In this subsection,
we prove Theorem~\ref{main:lowerbd} which shows
that the above upper bound on $|V(T)|$ is almost optimal:
there is a collection of graphs $\Gamma_m$ with $m$ vertices
such that if $\phi:T\to\Gamma_m$ is an immersion of a tree $T$
into $\Gamma_m$ and $\phi^*:G(\Gamma_m)\to G(T)$ is injective,
then
$$|V(T)| \geqslant  2^{m/4}.$$

\begin{theorem}\label{thm:ipl}
Let $\phi:\Lambda\to \Gamma$ be an immersion between finite graphs,
and let $F$ be a finite nonempty subset of $V(\Lambda)$.
If $\phi$ is $F$-surviving, then $\phi$ has IPL for $F$.
\end{theorem}

\begin{proof}
Let $p:\tilde\Gamma\to\Gamma$ be a covering
which extends $\phi$, i.e.
$\Lambda \leqslant \tilde\Gamma$ and $\phi=p|_\Lambda$.

Assume that $\phi$ is $F$-surviving, but does not have IPL for $F$.
Then there exist a vertex $v'\in F$ and an induced path
$\alpha=(v_0,v_1,\ldots,v_k)$ in $\Gamma$ with $v_0=\phi(v')$ and $k \geqslant 1$
such that if $\tilde \alpha=(v_0',v_1',\ldots,v_k')$ is the lift of $\alpha$
to $\tilde\Gamma$ starting from $v'=v_0'$, then
$v_k'\not\in \Lambda$ and $v_i'\in \Lambda$ for $0 \leqslant i \leqslant k-1$.
Consider the following reduced word
$$w=v_0\cdots v_{k-1}v_k v_{k-1}^{-1}\cdots v_0^{-1}.$$

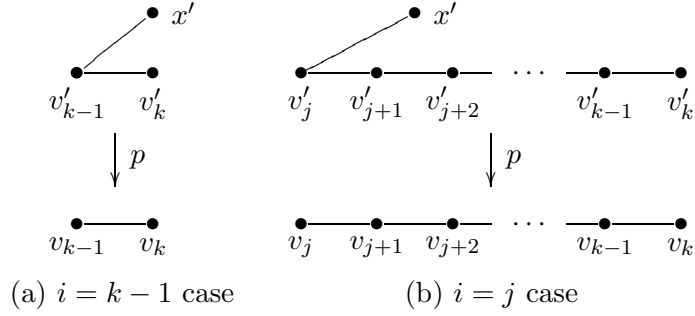
\begin{figure}
\begin{tabular}{ccc}
$\begin{xy}
( 0, 0) *{\bullet}; (10,0) *{\bullet} **@{-};
(0,-3) *{v_{k-1}};
(10,-3) *{v_{k}};
( 0, 20) *{\bullet}; (10, 20) *{\bullet} **@{-} ;
( 0, 20) *{\bullet}; (10, 28) *{\bullet} **@{-} ;
(0, 16) *{v_{k-1}'};
(10,16) *{v_k'};
(14,28) *{x'};
(5,5); (5,12) **@{-}    ?<*@{<}   ?(0.5) *!/^3mm/{p};
\end{xy}$
&&
$\begin{xy}
(0, 0) *{\bullet};
(10,0) *{\bullet} **@{-};
(20,0) *{\bullet} **@{-};
(25,0)  **@{-}; (30,0) *{\cdots}; (35,0);
(40,0) *{\bullet} **@{-};
(50,0) *{\bullet} **@{-};
(0,-3) *{v_{j}};
(10,-3) *{v_{j+1}};
(20,-3) *{v_{j+2}};
(40,-3) *{v_{k-1}};
(50,-3) *{v_{k}};
(0, 20) *{\bullet};
(10,20) *{\bullet} **@{-};
(20,20) *{\bullet} **@{-};
(25,20)  **@{-}; (30,20) *{\cdots}; (35,20);
(40,20) *{\bullet} **@{-};
(50,20) *{\bullet} **@{-};
(0, 16) *{v_{j}'};
(10,16) *{v_{j+1}'};
(20,16) *{v_{j+2}'};
(40,16) *{v_{k-1}'};
(50,16) *{v_{k}'};
(0, 20); (15,28) *{\bullet} **@{-};
(19,28) *{x'};
(25,5); (25,12) **@{-}    ?<*@{<}   ?(0.5) *!/^3mm/{p};
\end{xy}
$\\[1em]
(a) $i=k-1$ case &&
(b) $i=j$ case
\end{tabular}
\caption{Pictures for the proof of Theorem~\ref{thm:ipl}}
\label{fig:spl}
\end{figure}

\medskip

\noindent\emph{Claim.}\ \emph{For $i=0,\ldots,k-1$, we have
\begin{equation}\label{eq:ipl}
{\operatorname{supp}}(\phi^*(v_{i+1}\cdots  v_{k-1}v_k v_{k-1}^{-1}\cdots v_{i+1}^{-1}))
\cap {\operatorname{Lk}}_\Lambda(v_i')=\emptyset.
\end{equation}
In particular,
$v_i'\not\in {\operatorname{supp}}(\phi^*(v_iv_{i+1}\cdots  v_{k-1}v_k v_{k-1}^{-1}
\cdots v_{i+1}^{-1}v_i^{-1}))$.
}

\medskip

\begin{proof}[Proof of Claim]
Notice that the equality~\eqref{eq:ipl} implies that $v_i'$ commutes with each element in
${\operatorname{supp}}(\phi^*(v_{i+1}\cdots  v_{k-1}v_k v_{k-1}^{-1}\cdots v_{i+1}^{-1}))$,
hence $v_i'\not\in {\operatorname{supp}}(\phi^*(v_iv_{i+1}\cdots  v_{k-1}v_k v_{k-1}^{-1}
\cdots v_{i+1}^{-1}v_i^{-1}))$.
We will prove \eqref{eq:ipl} by using reverse induction on $i=0,\ldots,k-1$.

For the case $i=k-1$, assume that there exists
$x'\in {\operatorname{supp}}(\phi^*(v_k))\cap{\operatorname{Lk}}_\Lambda(v_{k-1}')$.
See Figure~\ref{fig:spl}(a).
Since $x'\in {\operatorname{Lk}}_\Lambda(v_{k-1}')$,
$\{v_{k-1}',x'\}$ is an edge in $\Lambda$.
Since $x'\in{\operatorname{supp}}(\phi^*(v_k)) \subseteq p^{-1}(v_k)$,
the edge $\{v_{k-1}',x'\}$ maps to the edge $\{v_{k-1},v_k\}$ by $p$.
Since $p$ is a covering, $v_k'=x'$ by the unique path lifting property.
This is a contradiction because $x'\in\Lambda$ but
$v_k'\not\in\Lambda$.
Therefore ${\operatorname{supp}}(\phi^*(v_k))\cap{\operatorname{Lk}}_\Lambda(v_{k-1}')=\emptyset$.

Now, assume that the claim is true for $i=j+1$ for some $0 \leqslant j \leqslant k-2$, hence
\begin{equation}\label{eq:ipl2}
v_{j+1}'\not\in {\operatorname{supp}}(\phi^*(v_{j+1}\cdots
 v_{k-1}v_k v_{k-1}^{-1}\cdots v_{j+1}^{-1})).
\end{equation}
Assume that there exists
$x'\in{\operatorname{supp}}(\phi^*(v_{j+1}\cdots  v_{k-1}v_k v_{k-1}^{-1}\cdots v_{j+1}^{-1}))
\cap{\operatorname{Lk}}_\Lambda(v_{j}')$.
See Figure~\ref{fig:spl}(b). Then
$$x'\ne v_{j+1}'$$
by \eqref{eq:ipl2}.
Since
$p(x')\in{\operatorname{supp}}(v_{j+1}\cdots  v_{k-1}v_k v_{k-1}^{-1}\cdots v_{j+1}^{-1})
=\{v_{j+1},\ldots,v_k\}$,
we have $p(x')=v_l$ for some $j+1 \leqslant l \leqslant k$.
Since $x'\in{\operatorname{Lk}}_\Lambda(v_{j}')$, $\{v_j',x'\}$ is an edge in $\Lambda$,
hence $\{v_j,v_l\}=\{p(v_j'), p(x')\}$ is an edge in $\Gamma$.
If $l \geqslant j+2$, this contradicts
the assumption that $\alpha$ is an induced path.
Therefore $l=j+1$.
Then $x'=v_{j+1}'$ by the unique path lifting property as before.
This is a contradiction to $x'\ne v_{j+1}'$.
This completes the proof of Claim.
\end{proof}

By the above claim, $v'=v_0'\not\in{\operatorname{supp}}(\phi^*(w))$,
hence $\phi$ is not $v'$-surviving.
This contradicts the assumption that $\phi$ is $F$-surviving because
$v'\in F$.
\end{proof}

\begin{figure}
\begin{tabular}{ccc}
$\begin{xy}
( 3, 0)*{\bullet};
( 3, 0); (10,-4) *{\bullet} **@{-} ;
( 3, 0); (10, 4) *{\bullet} **@{-} ;
(10,-4); (20,-4) *{\bullet} **@{-} ;
(10,-4); (20, 4) *{\bullet} **@{-} ;
(10, 4); (20, 4) *{\bullet} **@{-} ;
(10, 4); (20,-4) *{\bullet} **@{-} ;
(20,-4); (30,-4) *{\bullet} **@{-} ;
(20,-4); (30, 4) *{\bullet} **@{-} ;
(20, 4); (30,-4) *{\bullet} **@{-} ;
(20, 4); (30, 4) *{\bullet} **@{-} ;
(30,-4); (40,-4) *{\bullet} **@{-} ;
(30,-4); (40, 4) *{\bullet} **@{-} ;
(30, 4); (40,-4) *{\bullet} **@{-} ;
(30, 4); (40, 4) *{\bullet} **@{-} ;
( 0, 0) *{v_0};
(10,-6) *{v_1};
(20,-6) *{v_2};
(30,-6) *{v_3};
(40,-6) *{v_4};
(10, 6) *{u_1};
(20, 6) *{u_2};
(30, 6) *{u_3};
(40, 6) *{u_4};
\end{xy}$
&\qquad\qquad&
$\begin{xy}
( 3, 0)*{\bullet};
( 3, 0); (10,-4) *{\bullet} **@{-} ;
( 3, 0); (10, 4) *{\bullet} **@{-} ;
(10,-4); (20,-4) *{\bullet} **@{-} ;
(10,-4); (20, 4) *{\bullet} **@{-} ;
(10, 4); (20, 4) *{\bullet} **@{-} ;
(10, 4); (20,-4) *{\bullet} **@{-} ;
(20,-4); (30,-4) *{\bullet} **@{-} ;
(20,-4); (30, 4) *{\bullet} **@{-} ;
(20, 4); (30,-4) *{\bullet} **@{-} ;
(20, 4); (30, 4) *{\bullet} **@{-} ;
(30,-4); (40,-4) *{\bullet} **@{-} ;
(30,-4); (40, 4) *{\bullet} **@{-} ;
(30, 4); (40,-4) *{\bullet} **@{-} ;
(30, 4); (40, 4) *{\bullet} **@{-} ;
(40,-4); (47, 0) *{\bullet} **@{-} ;
(40, 4); (47, 0) *{\bullet} **@{-} ;
( 0, 0) *{v_0};
(10,-6) *{v_1};
(20,-6) *{v_2};
(30,-6) *{v_3};
(40,-6) *{v_4};
(10, 6) *{u_1};
(20, 6) *{u_2};
(30, 6) *{u_3};
(40, 6) *{u_4};
(50, 0) *{v_5};
\end{xy}$\\[2em]
(a) $\Gamma_m$ with $m=9$ &&
(b) $\Gamma_m$ with $m=10$
\end{tabular}
\caption{The graph $\Gamma_m$ in Theorem~\ref{main:lowerbd}}
\label{fig:lowerbd}
\end{figure}
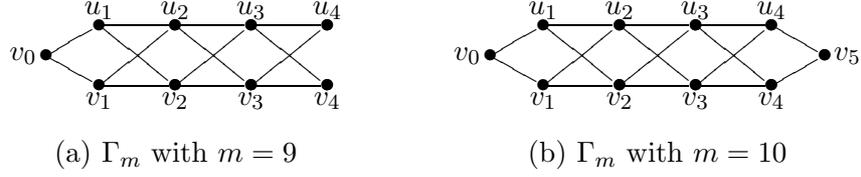

\begin{theorem}\label{main:lowerbd}
For each integer $m \geqslant 2$, there exists a graph  $\Gamma_m$ with $m$ vertices
such that if $\phi:T\to\Gamma_m$ is an immersion of a finite tree $T$ into $\Gamma_m$
and $\phi^*:G(\Gamma_m)\to G(T)$ is injective,
then $|V(T)| \geqslant 2^{m/4}$.
\end{theorem}

\begin{proof}
The theorem is obvious for $m=2$.
For $m \geqslant 3$, let $\Gamma_m$ be the graph defined as
\begin{align*}
V(\Gamma_{2k+1})&=\{u_i,v_i:1 \leqslant i \leqslant k\}\cup \{v_0\},\\
V(\Gamma_{2k+2})&=\{u_i,v_i:1 \leqslant i \leqslant k\}\cup \{v_0,v_{k+1}\},\\
E(\Gamma_{2k+1})&= \{ \{v_i,v_{i+1}\}, \{u_i,u_{i+1}\}, \{v_i,u_{i+1}\}, \{u_i,v_{i+1}\}: 1 \leqslant i \leqslant k-1\}\\
&\qquad\cup \{ \{v_0,u_1\}, \{v_0,v_1\}\},\\
E(\Gamma_{2k+2})&=E(\Gamma_{2k+1})\cup \{ \{v_k,v_{k+1}\}, \{u_k,v_{k+1}\} \}.
\end{align*}
In other words, there is an edge between every pair of vertices with adjacent indices.
See Figure~\ref{fig:lowerbd} for the cases $m=9,10$.

Let $\phi:T\to\Gamma_m$ be an immersion of a tree $T$ into $\Gamma_m$
such that $\phi^*:G(\Gamma_m)\to G(T)$ is injective.
Notice that $\phi$ is surjective on the sets of vertices.

Because $T$ is a tree and $\phi$ is an immersion,
we can consider $\phi$ as a restriction of the universal cover
$p:\tilde\Gamma_m\to\Gamma_m$ to an induced subtree $T$,
i.e. $\phi=p|_T$.

Take any subset $F$ of $V(\tilde\Gamma_m)$ such that
$|F|=m$ and $p(F)=V(\Gamma_m)$.
(For example, $F$ can be chosen to be the set of vertices
of a lift of a maximal tree in $\Gamma_m$
to $\tilde\Gamma_m$.)
Let $\Sigma$ be the set of all deck transformations
$\sigma:\tilde\Gamma_m\to\tilde\Gamma_m$ such that
$\sigma(T)\cap F\ne\emptyset$, and let
$T_1$ be the induced subgraph of $\tilde\Gamma_m$ on
$$
\bigcup_{\sigma\in\Sigma}\sigma(V(T)).
$$
Then $F \subseteq  V(T_1)$, and
$\phi_1=p|_{T_1}:T_1\to\Gamma_m$ is $F$-surviving by Lemma~\ref{lem:deck},
hence it has IPL for $F$ by Theorem~\ref{thm:ipl}.
By Remark~\ref{rmk:deck-size},
$$|V(T_1)| \leqslant |\Sigma|\cdot|V(T)| \leqslant |V(T)|^2.$$

Now, we will show that $|V(T_1)| \geqslant 2^{m/2}$.

First, assume that $m=2k+1$.
The induced paths $\alpha$
starting from $v_0$ in $\Gamma_{2k+1}$ are of the form
$\alpha=(v_0,x_1,x_2,\ldots,x_l)$, where $0 \leqslant  l \leqslant k$ and
$x_i\in\{u_i,v_i\}$ for $1 \leqslant  i \leqslant l$.
There are $2^l$ such paths $\alpha$ of length $l$, and
all their lifts to $\tilde\Gamma_{2k+1}$ starting from $v_0'$
are contained in $T_1$.
If $\alpha_1$ and $\alpha_2$ are distinct induced paths
starting from $v_0$ in $\Gamma_{2k+1}$, then their lifts
to $\tilde\Gamma_{2k+1}$ starting from $v_0'$ have distinct
endpoints because $\tilde\Gamma_{2k+1}$ is the universal cover.
Hence
$$
|V(T_1)| \geqslant 1+2+\cdots+2^k=2^{k+1}-1=2^{(m+1)/2}-1.
$$

When $m=2k+2$, the same argument as above gives
\begin{align*}
|V(T_1)|
& \geqslant 1+2+\cdots+2^k+2^k=2^{k+1}-1+2^k\\
&=\frac32\cdot 2^{k+1}-1
=\frac32\cdot 2^{m/2}-1.
\end{align*}
Therefore, for $m \geqslant 3$, we have $|V(T_1)| \geqslant 2^{m/2}$,
and hence $|V(T)| \geqslant |V(T_1)|^{1/2} \geqslant 2^{m/4}$.
\end{proof}

\subsection{Embedding of RAAGs on cycle graphs into RAAGs on path graphs}
\label{ssec:cdk}

Let $C_m$ and $P_n$ denote the cycle
and path graphs on $m$ and $n$ vertices, respectively.
We denote their vertices by
$$
V(C_m)=\{v_0,v_1,\ldots,v_{m-1}\}\quad\mbox{and}\quad
V(P_n)=\{v_0',v_1',\ldots,v_{n-1}'\}.
$$
The edge sets of $C_m$ and $P_n$ are
\begin{align*}
E(C_m)&=\{\{v_i,v_{i+1}\}:0 \leqslant i \leqslant m-2\}\cup \{\{v_{m-1},v_0\}\},\\
E(P_n)&=\{\{v_j',v_{j+1}'\}:0\leqslant j\leqslant n-2\}.
\end{align*}

Let $\phi_{n,m}:P_n\to C_m$ be the immersion defined by
$$\phi_{n,m}(v_j')=v_{j\bmod m}.$$
Regard $P_n$ as an induced subgraph of the universal cover
$\tilde C_m$ of $C_m$ which is the bi-infinite path graph.
Then the map $\phi_{n,m}$ is the restriction of the universal cover
$p:\tilde C_m\to C_m$ to $P_n$.

\begin{figure}
$$\begin{xy}
(0, 30) *{\bullet}; (10,30) *{\bullet}  **@{-} ; (15,30) **@{-};
(20,30) *{\cdots}; (25,30) ;
(30,30) *{\bullet}  **@{-} ; (40,30) *{\bullet}  **@{-} ; (50,30) *{\bullet}  **@{-} ;
(55,30)  **@{-} ; (60,30) *{\cdots}; (65,30);
(70,30) *{\bullet}  **@{-} ;
(0, 27) *{v_2''};
(10,27) *{v_3''};
(30,27) *{v_{m-1}''};
(40,27) *{v_0'};
(50,27) *{v_1'};
(70,27) *{v_{m-1}'};
(78,30) *+!L{\txt{$P_{2m-2}$}};
(10,20) *{\bullet}; (15,20) **@{-};
(20,20) *{\cdots}; (25,20) ;
(30,20) *{\bullet}  **@{-} ; (40,20) *{\bullet}  **@{-} ; (50,20) *{\bullet}  **@{-} ;
(55,20)  **@{-} ; (60,20) *{\cdots}; (65,20);
(70,20) *{\bullet}  **@{-} ;
(10,17) *{v_3''};
(30,17) *{v_{m-1}''};
(40,17) *{v_0'};
(50,17) *{v_1'};
(70,17) *{v_{m-1}'};
(78,20) *+!L{\txt{$P_{2m-3}$}};
(40,10) *{\bullet};
(10, 0) *{\bullet} **@{-}; (20, 0) *{\bullet} **@{-}; (30, 0) *{\bullet} **@{-};
(35, 0) **@{-}; (40, 0) *{\cdots}; (45,0);
(50, 0) *{\bullet}  **@{-} ; (60,0) *{\bullet} **@{-};(70, 0) *{\bullet}  **@{-};
(40,10) **@{-};
(40, 7) *{v_0};
(10,-3) *{v_1};
(20,-3) *{v_2};
(30,-3) *{v_3};
(50,-3) *{v_{m-3}};
(60,-3) *{v_{m-2}};
(70,-3) *{v_{m-1}};
(78, 0) *+!L{\txt{$C_m$}};
\end{xy}
$$
\caption{The graphs $P_{2m-2}$, $P_{2m-3}$ and $C_m$}
\label{fig:cdk}
\end{figure}
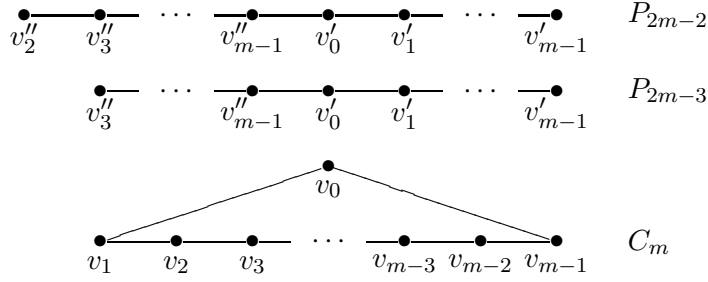

In \cite{CDK13}, M.{} Casals-Ruiz, A.{} Duncan and I.{} Kazachkov
showed that $\phi_{8,5}^*:G(C_5)\to G(P_8)$ is injective.
Using the induced path lifting property,
we generalize their result as follows.

\begin{theorem}\label{main:cdk}
For each $m \geqslant 3$, $\phi_{n,m}^*:G(C_m)\to G(P_n)$ is injective
if and only if $n \geqslant  2m-2$.
\end{theorem}

\begin{proof}
It suffices to show that $\phi_{2m-2,m}^*$ is injective and that
$\phi_{2m-3,m}^*$ is not injective,
because $\ker\phi_{l,m}^* \subseteq \ker\phi_{k,m}^*$ for $1\leqslant k\leqslant l$
by Lemma~\ref{lem:KK9}.

First, we will show that $\phi_{2m-2,m}^*$ is injective.
Let $\phi=\phi_{2m-2,m}:P_{2m-2}\to C_m$.
Renaming the vertices, we may assume that
$$V(P_{2m-2})=\{v_0',v_1',\ldots,v_{m-1}', v_2'',\ldots,v_{m-1}''\}$$
as in Figure~\ref{fig:cdk}
so that $\phi$ maps $v_i'$ and $v_i''$ to $v_i$ for each $i$.
Notice that ${\operatorname{Lk}}_{C_m}(v_0)=\{v_1,v_{m-1}\}$
and ${\operatorname{Lk}}_{P_{2m-2}}(v_0')=\{v_1',v_{m-1}''\}$.
Let
\begin{alignat*}{2}
\Gamma_1&=C_m{\backslash} v_0,
     &\Lambda_1&=P_{2m-2}{\backslash} \phi^{-1}(v_0)=P_{2m-2}{\backslash} v_0',\\
\Gamma_2&=C_m{\backslash}\{v_0,v_1\},\qquad
     &\Lambda_2&=P_{2m-2}{\backslash} \phi^{-1}(\{v_0,v_1\})
     =P_{2m-2}{\backslash}\{v_0',v_1'\}.
\end{alignat*}
Then $\phi(\Lambda_i)=\Gamma_i$ for $i=1,2$.
Let $\phi_i=\phi(\Lambda_i,\Gamma_i)$ for $i=1,2$.

The graph $\Lambda_1$ has two components.
One of them is the path graph on $\{v_1',\ldots,v_{m-1}'\}$,
and it is isomorphic to $\Gamma_1$ under $\phi_1$.
Hence $\phi_1$ is $v_1'$-surviving and $\phi_1^*$ is injective.
Similarly, $\phi_2$ is $v_{m-1}''$-surviving (and $\phi_2^*$ is injective).
Therefore $\phi^*=\phi_{2m-2,m}^*$ is injective by Lemma~\ref{lem:link-s}.

\smallskip
Now, we will show that $\phi_{2m-3,m}^*$ is not injective.
Let $\phi=\phi_{2m-3,m}:P_{2m-3}\to C_m$.
Assume that $\phi^*$ is injective.
Renaming the vertices, we may assume that
$$V(P_{2m-3})=\{v_0',v_1',\ldots,v_{m-1}',v_3'',\ldots,v_{m-1}''\}$$
as in Figure~\ref{fig:cdk} so that $\phi$ maps $v_i'$ and $v_i''$ to $v_i$ for each $i$.
Let $\Sigma$ be the set of all deck transformations $\sigma:\tilde C_m\to \tilde C_m$ such that
$\sigma(P_{2m-3})\cap\{v_0'\}\ne\emptyset$.
It is easy to see that $\Sigma=\{ \operatorname{Id} \}$.
Then $\phi:P_{2m-3}\to C_m$ is $v_0'$-surviving by Lemma~\ref{lem:deck},
hence $\phi$ has IPL for $v_0'$ by Theorem~\ref{thm:ipl}.
Thus there is a path in $P_{2m-3}$ that is a lift
of the induced path $(v_0,v_{m-1},v_{m-2},\ldots,v_2)$, which is impossible.
Therefore $\phi^*=\phi_{2m-3,m}^*$ is not injective.
\end{proof}

\section*{Acknowledgement}
The authors are grateful to Sang-hyun Kim for the helpful conversation,
and to Young Soo Kwon, Sang-il Oum and Paul Seymour for
making the example for Theorem~\ref{main:lowerbd}.
The first author was partially supported by Basic Science Research Program
through the National Research Foundation of Korea (NRF)
funded by the Ministry of Science, ICT \& Future Planning (NRF-2012R1A1A3006304).
The second author was partially supported by Basic Science Research Program
through the National Research Foundation of Korea (NRF)
funded by the Ministry of Education (NRF-2013R1A1A2007523).

\end{document}